\documentclass[11pt]{article}
\usepackage{setspace}
\usepackage{amsmath,amssymb,amsfonts,amsthm}
\newcounter{item}[section]
\newcounter{kirshr}
\newcounter{kirsha}
\newcounter{kirshb}
\newtheorem{theorem}{Theorem}[section]

\newtheorem{corollary}[theorem]{Corollary}
\theoremstyle{definition}
\newtheorem{remark}[theorem]{Remark}

\newtheorem{definition}[theorem]{Definition}

\def\(R)RA{{\bf (R)RA}}

\begin{document}

\begin{center}
\Large{Distinguishing Models by Formulas and the Number of Countable Models}\\[1cm]
\small{Mohammad Assem}
\begin{abstract} We indicate a way of distinguishing between structures, for which, we call two structures \emph{distinguishable.} Roughly, being distinguishable means that they differ in the number of realizations each gives for some formula. Being non-distinguishable turns out to be an interesting equivalence relation that is weaker than isomorphism and stronger than elementary equivalence. We show that this equivalence relation is Borel in a Polish space that codes countable structures. It then follows, without assuming the Continuum Hypothesis, that for any first order theory in a countable language, if it has an uncountable set of countable models that are pairwise distinguishable, then actually it has such a set of continuum size. We show also, as an easy consequence of our results, that Vaught's conjecture holds for the language with only one unary relation symbol.
\end{abstract}
\emph{MSC classification: 03C15}
\end{center}

\vspace{1cm}

\doublespacing
The subject of this paper relates to the famous conjecture of Vaught. Vaught's conjecture states that the number of countable models (non-isomorphic) of a first-order theory is either countable or continuum. Settling the conjecture stills an open problem since 1961. It turns out interesting to study the conjecture for equivalence relations other than isomorphism between models. For example, it is well known that the conjecture holds if we consider elementary equivalence. Here we show that the conjecture holds for some stronger equivalence relation. We also give an example in which  our equivalence relation is as strong as isomorphism.

\vspace{1cm}

\section*{Basics} Our system of notation is mostly standard, but the following list may be useful. Throughout, both $\omega$ and $\mathbb{N}$ denote the set of natural numbers $\{0,1,2,\ldots\}$ and for every $n\in\omega$ we have $n=\{0,\ldots,n-1\}.$ Let $A$ and $B$ be sets. Then ${}^AB$ denotes the set of functions whose domain is $A$ and whose range is a subset of $B.$ In addition, $|A|$ denotes the cardinality of $A$ and $\mathcal{P}(A)$ denotes the power set of $A$, that is, the set of all subsets of $A$. If $f:A\longrightarrow B$ is a function and $X\subseteq A,$ then $f^*(X)=\{f(x):x\in X\}.$ Moreover, $f^{-1}:\mathcal{P}(B)\longrightarrow\mathcal{P}(A)$ acts between the powersets.
Let ${}^{<\omega}\omega$ denote the set of finite sequences of naturals.
\vspace{1cm}

 We recall now some notions.
Let $L$ denote a non-empty countable relational language (this entails no loss of generality): $L=(R_i)_{i \in I}$ where $I$ is a non-empty countable index set and $R_i$ is an $n_i$-ary relation symbol. Denote by $X_L$ the space $$X_L=\prod_{i \in I} 2^{(\mathbb{N}^{n_i})}.$$ We view the space $X_L$ as the space of countably infinite \emph{L-structures}, identifying every $x=(x_i)_{i \in I} \in X_L$ with the structure $$\mathcal{U}_x=(\mathbb{N},R_i^{\mathcal{U}_x})_{i\in I},$$ where $R_i^{\mathcal{U}_x}(\bar{s})\Leftrightarrow x_i(\bar{s})=1, \mbox{ for }\bar{s}\in \mathbb{N}^{n_i}$. We call $x$ the \textit{code} for the $L$-structure $\mathcal{U}_x$.

$L_{\omega_1\omega}$ is an infinitary language. There is a countably infinite set of \emph{variables}. The \emph{atomic formulas} are those of the form $R_i(v_{1},\ldots,v_{n_i}),$ where $R_i$ is a relation symbol of $L$ and $v_j$'s are arbitrary variables. The \emph{formulas} of $L_{\omega_1\omega}$ are built up from these formulas using negation, (first order) quantifiers, and finite or countably infinite conjunctions and disjunctions (of sets of formulas whose variables come from a fixed finite set). For $\varphi$, a formula of $L_{\omega_1\omega}$, let $\Delta \varphi$ denote the set of its free variables. A formula of $L_{\omega_{1}\omega}$ that has no free variables is called a \emph{sentence}.

A \emph{fragment} $F$ of $L_{\omega_{1}\omega}$ is a set of formulas in $L_{\omega_{1}\omega}$ containing all atomic formulas, closed under subformulas, negation, quantifiers and finite conjunctions and disjunctions.

\begin{definition}
For $\varphi(=\varphi(\bar{v}))$ a formula of $L_{\omega_{1}\omega}$ and $\bar{s}$ a finite sequence from $\mathbb{N}$ of appropriate length (i.e, $\bar{s}\in{}^{|\Delta\varphi|}\omega$), let $$Mod(\varphi,\bar{s})=\{x\in X_L: \mathcal{U}_x\models\varphi[\bar{s}]\},$$ where $\varphi[\bar{s}]$ denotes the sentence obtained from the formula $\varphi(\bar{v})$ by substituting $\bar{s}$ for the free variables. (If $\varphi$ is a sentence, we write $Mod(\varphi)$ for $Mod(\varphi,())$).
\end{definition}

Let $t_F$ be the topology on $X_L$ generated by $\mathcal{B}_F=\{Mod(\varphi,\bar{s}):\varphi\in F, \bar{s}\in{}^{|\Delta\varphi|}\omega\}.$ By a result of Sami (see \cite{Sami}), if $F$ is a countable fragment, then $t_F$ is a Polish topology on $X_L.$

\begin{definition}Let $F$ be a fragment of $L_{\omega_{1}\omega}$ and let $A\subseteq F$. We say that $x,y\in X_L$ (or \emph{their corresponding structures}) are \textit{distinguishable in} $A$, if there is $\varphi\in A$ such that $|\varphi^x|\neq|\varphi^y|,$ where $\varphi^x=\{\bar{s}\in{}^{|\Delta\varphi|}\omega:\mathcal{U}_x\models\varphi[\bar{s}]\}$. (It is clear that if two structures are distinguishable, then they are non-isomorphic). Notice that, if $\varphi$ is a sentence, then for all $x$, either $\varphi^x$ is empty or else contains only the empty sequence.
\end{definition}

\section*{The Equivalence Relation(s)}Let $F$ be a fragment of $L_{\omega_{1}\omega}$ and $A\subseteq F$. Let $E_A$ be the following equivalence relation on $X_L$:
$$E_A=\{(x,y)\in X_L\times X_L: \mbox{ for all }\varphi\in A, \mbox{ }|\varphi^x|=|\varphi^y|\}.$$
Let us now state the main theorem of the paper.
\begin{theorem}[Main theorem]\label{main}
For $F$ a fragment of $L_{\omega_{1}\omega}$ and $A\subseteq F$ countable, $E_A$ is Borel in the product topology $(X_L,t_F)\times(X_L,t_F)$.
\end{theorem}

It should be remarked that the proofs in this paper hold for languages with or without equality. We also present, in Remark \ref{equality} a shorter  proof of our Main Theorem using languages with equality.

Our next proof of Theorem \ref{main} needs the development of some tools that will enable us to express \textbf{appropriately} sentences talking about sets, sentences like ``the following two sets $X,Y$ are of the same size".

Let $\mu$ be a bijection between $\mathbb{N}$ and ${}^{<\omega}\omega$. Then we can easily see that for a set $X\subseteq {}^{<\omega}\omega$:
$$X\mbox{ is infinite \textit{iff}  } (\forall n)(\exists m>n) \mu(m)\in X.$$
Suppose now for $X,Y$ subsets of ${}^{<\omega}\omega$, we want to say that $|X|=|Y|<\omega$. We show that the following sentence states appropriately what we want: ``
\textit{There is } $n<\omega$ \textit{ and there are injective maps} $f,g:n\longrightarrow {}^{<\omega}\omega$ \textit{such that} $(\forall t)[t\in X\Leftrightarrow g(f^{-1}(t))\in Y]$ and $(\forall t)[t\in Y\Leftrightarrow f(g^{-1}(t))\in X]$".

First we show that $|X|=|Y|<\omega$ iff
\textit{There is } $n<\omega$ \textit{ and there are injective maps} $f,g:n\longrightarrow {}^{<\omega}\omega$ \textit{such} \textit{that} $f^*(g^{-1}(Y))=X$ and $g^*(f^{-1}(X))=Y$. Indeed, let $Inj(n,{}^{<\omega}\omega)$ denote the set of all injections from $n$ into ${}^{<\omega}\omega$. Then,
\begin{align*}
 |X|=|Y|\in\omega &\Longleftrightarrow (\exists n)(\exists f,g\in Inj(n,{}^{<\omega}\omega))(f^*(n)=X\wedge g^*(n)=Y)\\
  &\Longrightarrow (\exists n)(\exists f,g\in Inj(n,{}^{<\omega}\omega))(f^*(g^{-1}(Y))=X\wedge \\ &g^*(f^{-1}(X))=Y) \\
 &\Longrightarrow (\exists n)(\exists f,g\in Inj(n,{}^{<\omega}\omega))(X\subseteq f^*(n)\wedge Y\subseteq g^*(n)\\&\wedge g^{-1}(Y)=f^{-1}(X))\\
 &\Longrightarrow (\exists n)(\exists f,g\in Inj(n,{}^{<\omega}\omega))(|X|=|f^{-1}(X)|\wedge |Y|=\\&|g^{-1}(Y)|\wedge |g^{-1}(Y)|=|f^{-1}(X)|)\\
 &\Longrightarrow |X|=|Y|\in\omega.
\end{align*}
Note that the second implication is the desired expression. It remains to analyze ``$f^*(g^{-1}(Y))=X$" (and similarly will be ``$g^*(f^{-1}(X))=Y$").

For $h$, an injective map from a subset of $\omega$ into $\omega$, let $h^{-1}$ denote also the map from $Range(h)$ to $\omega$ that sends $t\in Range(h)$ to the unique element in $h^{-1}(\{t\})$.
Remark now that, for $f$ and $g$ like above, because they are injective we have:
\begin{align*}
f^*(g^{-1}(Y))=X &\Longleftrightarrow (\forall t)[t\in X\Leftrightarrow (\exists s)(s\in Y\wedge f(g^{-1}(s))=t)]\\
&\Longleftrightarrow (\forall t)[t\in X\Leftrightarrow (\exists s)(s\in Y\wedge g^{-1}(s)=f^{-1}(t))]\\
&\Longleftrightarrow (\forall t)[t\in X\Leftrightarrow (\exists s)(s\in Y\wedge s=g(f^{-1}(t)))]\\
&\Longleftrightarrow (\forall t)[t\in X\Leftrightarrow g(f^{-1}(t))\in Y].
\end{align*}

Now after knowing the appropriate way for expressing ``$|X|=|Y|$", we are ready to prove our main theorem.

\begin{proof} We use the above work to show that $E_A$ is Borel. The following are just direct steps carried out in detail. We have
\begin{align*}
 E_A &= \{(x,y): (\forall\varphi\in A) (|\varphi^x|=|\varphi^y|)\}\\
  &= \bigcap_{\varphi\in A}\{(x,y): |\varphi^x|=|\varphi^y|\}\\
  &=\bigcap_{\varphi\in A}\Big[\{(x,y): |\varphi^x|=|\varphi^y|\in\omega\}\cup  \{(x,y): |\varphi^x|,|\varphi^y|\text{ are both infinite}\}\Big]\\
  &=\bigcap_{\varphi\in A}\Big[\{(x,y):(\exists n)(\exists f,g\in Inj(n,{}^{<\omega}\omega))(f^*(g^{-1}(\varphi^y))=\varphi^x\wedge \\& g^*(f^{-1}(\varphi^x))=\varphi^y)\}\cup
  \{(x,y):(\forall n)(\exists m,k>n)(\mu(m)\in \varphi^x\wedge \mu(k)\in \varphi^y)\}\Big]\\
  &=\bigcap_{\varphi\in A}\Big[\bigcup_n\bigcup_{f,g\in Inj(n,{}^{<\omega}\omega)}\{(x,y):f^*(g^{-1}(\varphi^y))=\varphi^x\wedge g^*(f^{-1}(\varphi^x))=\varphi^y\}\cup\\ &\bigcap_n\bigcup_{m,k>n}\{(x,y):\mu(m)\in \varphi^x\wedge \mu(k)\in \varphi^y\}\Big]\\
         \end{align*}
  \begin{align*}
  &=\bigcap_{\varphi\in A}\Big[\bigcup_n\bigcup_{f,g\in Inj(n,{}^{<\omega}\omega)}\{(x,y):(\forall t\in{}^{<\omega}\omega)(t\in \varphi^x \Leftrightarrow g( f^{-1}(t))\in \varphi^y)\wedge \\
  &(\forall t\in{}^{<\omega}\omega)(t\in \varphi^y \Leftrightarrow f( g^{-1}(t))\in \varphi^x)\}\cup \bigcap_n\bigcup_{m,k>n}\{(x,y):x\in Mod(\varphi,\mu(m))\wedge \\&y\in Mod(\varphi,\mu(k))\}\Big]\\
  &=\bigcap_{\varphi\in A}\Big[\bigcup_n\bigcup_{f,g\in Inj(n,{}^{<\omega}\omega)}\bigcap_{t\in{}^{<\omega}\omega}\{(x,y):(t\in \varphi^x \Leftrightarrow g( f^{-1}(t))\in \varphi^y)\wedge\\&(t\in \varphi^y \Leftrightarrow f( g^{-1}(t))\in\varphi^x)\}\cup\bigcap_n\bigcup_{m,k>n}(Mod(\varphi,\mu(m))\times Mod(\varphi,\mu(k)))\Big]\\
  &=\bigcap_{\varphi\in A}\Big[\bigcup_n\bigcup_{f,g\in Inj(n,{}^{<\omega}\omega)}\bigcap_{t\in{}^{<\omega}\omega}\{(x,y):[x\in Mod(\varphi,t) \Leftrightarrow  y\in Mod(\varphi,g( f^{-1}(t)))]\wedge\\&[y\in Mod(\varphi,t)  \Leftrightarrow x\in Mod(\varphi,f( g^{-1}(t)))]\}\cup\bigcap_n\bigcup_{m,k>n}(Mod(\varphi,\mu(m))\times \\& Mod(\varphi,\mu(k)))\Big]\\
  &=\bigcap_{\varphi\in A}\Big[\bigcup_n\bigcup_{f,g\in Inj(n,{}^{<\omega}\omega)}\bigcap_{t\in{}^{<\omega}\omega}\big(\{(x,y):x\in Mod(\varphi,t) \Leftrightarrow y\in Mod(\varphi,g( f^{-1}(t)))\}\cap\\&\{(x,y):y\in Mod(\varphi,t)  \Leftrightarrow x\in Mod(\varphi,f( g^{-1}(t)))\}\big)\cup\bigcap_n\bigcup_{m,k>n}[Mod(\varphi,\mu(m))\\&\times Mod(\varphi,\mu(k))]\Big]\\
  &=\bigcap_{\varphi\in A}\Big[\bigcup_n\bigcup_{f,g\in Inj(n,{}^{<\omega}\omega)}\bigcap_{t\in{}^{<\omega}\omega}\bigg(\{(x,y):[x\in Mod(\varphi,t) \wedge y\in Mod(\varphi,g( f^{-1}(t)))]\vee\\&[x\notin Mod(\varphi,t) \wedge y\notin Mod(\varphi,g( f^{-1}(t)))]\}\cap\{(x,y):[y\in Mod(\varphi,t) \wedge\\& x\in Mod(\varphi,f( g^{-1}(t)))]\vee [y\notin Mod(\varphi,t)  \wedge x\notin Mod(\varphi,f( g^{-1}(t)))]\}\bigg) \cup\\&\bigcap_n\bigcup_{m,k>n}[Mod(\varphi,\mu(m))\times Mod(\varphi,\mu(k))]\Big]\\
   \end{align*}
  \begin{align*}
  &=\bigcap_{\varphi\in A}\Big[\bigcup_n\bigcup_{f,g\in Inj(n,{}^{<\omega}\omega)}\bigcap_{t\in{}^{<\omega}\omega}\bigg(\Big([Mod(\varphi,t) \times Mod(\varphi,g( f^{-1}(t)))]\cup [Mod(\neg\varphi,t) \times \\&Mod(\neg\varphi,g( f^{-1}(t)))]\Big)\cap \Big([Mod(\varphi,f( g^{-1}(t)))\times Mod(\varphi,t)]\cup  [Mod(\neg\varphi,f( g^{-1}(t)))\times \\&Mod(\neg\varphi,t)]\Big)\bigg) \cup\bigcap_n\bigcup_{m,k>n}[Mod(\varphi,\mu(m))\times Mod(\varphi,\mu(k))]\Big].
  \end{align*}
\end{proof}

It follows now, from Theorem \ref{main}, that when $F$ is countable (so that $t_F$ is Polish), $E_A$ satisfies the Glimm-Effros Dichotomy (see \cite{HKL}). Thus, it has either countably many equivalence classes or else perfectly many classes.
\begin{corollary}\label{cor}
Let $T$ be an $L_{\omega_1\omega}$ sentence (this includes first order theories in  countable languages). If $T$ has an uncountable set of pairwise distinguishable (in any countable fragment of $L_{\omega_1\omega}$) countable models, then it has such a set of size $2^{\aleph_0}$ (and so has $2^{\aleph_0}$ non-isomorphic countable models).
\end{corollary}
\begin{proof}
Let $F$ be a countable fragment of $L_{\omega_1\omega}$. Let $M^T$ denote the Borel subset of $X_L$ that corresponds to models of $T$. $M^T$ is a standard Borel space with the topology induced by $t_F$. It follows, by a result of Silver (see \cite{Silver}), that the restriction of $E_F$ to $M^T$ (which is Borel) has either countably many equivalence classes or else perfectly many.
\end{proof}

It also follows from our results that Vaught's conjecture holds for the very simple language $\mathfrak{L}=\{R\}$ where $R$ is a unary relation symbol. To see this, we just need to notice that \emph{two $\mathfrak{L}$-structures are isomorphic iff their codes in $X_\mathfrak{L}=2^{\mathbb{N}}$ are equivalent w.r.t. $E_{\{R(v),\neg R(v)\}}$}. For more details, assume that for $x,y\in X_{\mathfrak{L}},$ $|R^x|=|R^y|$ and $|(\neg R)^x|=|(\neg R)^y|$ (note that $(\neg R)^x=\mathbb{N}\setminus R^x$). Then there are bijections $\tau:R^x\longrightarrow R^y$ and $\sigma:\mathbb{N}\setminus R^x\longrightarrow\mathbb{N}\setminus R^y.$ Clearly, $\theta=\tau\cup \sigma$ is a permutation of $\mathbb{N}$. Moreover, it is an isomorphism from $\mathcal{U}_x$ onto $\mathcal{U}_y$ because $$(\forall s\in\mathbb{N})(x(s)=1\Longleftrightarrow y(\theta(s))=1).$$ The last talk can also be restated in the language of Polish group actions. This gives us a simple example of a continuous action of the Polish group $S_\infty$ (the symmetry group of $\mathbb{N}$) on the Polish space $X_\mathfrak{L}$  with Borel orbit equivalence relation ($X_\mathfrak{L}$ is equipped with any of the the topologies $t_F$ defined above). Our action is $J:S_\infty\times 2^\mathbb{N}\longrightarrow2^\mathbb{N},$ defined by $$J(g,x)=y\Longleftrightarrow x\circ g^{-1}=y.$$ Recall that this action is continuous (see \cite{Sami}). Clearly, $J(g,x)=y$ iff $g$ is an isomorphism from $\mathcal{U}_x$ onto $\mathcal{U}_y.$ By the above discussion, the orbit equivalence relation is Borel.

\begin{remark}\label{equality}
Let $F$ be a countable fragment of $L_{\omega_1\omega}$ and $A\subseteq F.$ Assume our language $L$ with equality. Let  $L^*=L_0\cup L_1$ where $L_0$ and $L_1$ are disjoint copies of $L$.
Then, $X_{L^*}\cong X_{L}\times X_{L}$.

For each formula $\varphi\in A,$ let $\varphi^*$ be the
sentence $$\bigwedge_{n\in\omega}(\exists^n \bar{x})\varphi_0(\bar{x})\leftrightarrow (\exists^n \bar{x})\varphi_1(\bar{x})$$
where $\varphi_0,$ $\varphi_1$ are the copies of $\varphi$ in $L_0,L_1$ respectively, and $\exists^n$ is a shorthand for
``there exists at least $n$ tuples such that ..."

It is then immediate that  $(x,y)\in E_A$ iff the structure $\mathcal{U}$ of $L^*$ such
that $\mathcal{U}|_{L_0}=\mathcal{U}_x$ and $\mathcal{U}|_{L_1}=\mathcal{U}_y$ satisfies $\bigwedge_{\varphi\in A}\varphi^*.$
This means that our equivalence relation between structures corresponds to the Borel subset of $X_{L^*}$ of models of
the formula $\bigwedge_{\varphi\in A}\varphi^*.$
\end{remark}

\vspace{8cm}
\footnotesize
Department of Mathematics, Faculty of Science, Cairo University, Giza, Egypt.

E-mail: moassem89@gmail.com
\end{document}